\documentclass[11pt]{amsproc}
 \usepackage[margin=1in]{geometry}
\usepackage{setspace,fullpage}
\usepackage{graphicx}
\usepackage[nice]{nicefrac}
\usepackage{amssymb}
\usepackage{multirow}
\usepackage{array}
\usepackage{amsmath,amsthm,amscd,amssymb, mathrsfs}

\usepackage{latexsym}

\numberwithin{equation}{section}

\theoremstyle{plain}
\newtheorem{theorem}{Theorem}[section]
\newtheorem{lemma}[theorem]{Lemma}

\newtheorem{proposition}[theorem]{Proposition}

\theoremstyle{definition}

\theoremstyle{remark}

\title{A bilinear approach to the finite field restriction problem}
\author{Mark Lewko}

\address{Lebanon, New Hampshire USA}
\email{mlewko@gmail.com}
\begin{document}

\begin{abstract}
Let $P$ denote the $3$-dimensional paraboloid over a finite field of odd characteristic in which $-1$ is not a square. We show that the Fourier extension operator associated with $P$ maps $L^2$ to $L^{r}$ for $r > \frac{32}{9} \approx3.555$. In contrast with much of the recent progress on this problem, our argument does not use state-of-the-art incidence estimates but rather proceeds by obtaining estimates on a related bilinear operator. These estimates are based on a geometric result that, roughly speaking, states that a set of points in the finite plane $F^2$ can be decomposed as a union of sets each of which either contains a controlled number of rectangles or a controlled number of trapezoids. 

\textbf{Erratum.} The published version of this manuscript claimed the stronger range $r>\frac{24}{7}$. The argument up to the final exponent calculation remains valid, but the claimed range was misstated due to a computation error.
\end{abstract}

\maketitle

\section*{Erratum}
In the published version of this paper, Theorem \ref{thm:Main} was stated with the exponent range $r>\frac{24}{7}$. This range was misstated due to an error in the final dominance calculation in Section 6. The manuscript has been revised to reflect the corrected argument. The bilinear estimate and the geometric estimate, Proposition \ref{prop:TboundGF}, are unchanged; what changes is the final bookkeeping of the dominant terms.

The error occurs after Section 6 obtains
\[
\|\widehat g\|_{L^2(P,d\sigma)}
\lesssim |G|^{\frac12}
+ |G|^{\frac{11}{16}}|F|^{-\frac18}
+ |F|^{-\frac{3}{16}}|G|^{\frac{13}{16}}
+ |G|^{\frac{17}{24}}
+ |F|^{\frac18}|G|^{\frac58}.
\]
In the published version, the term $|F|^{-\frac{3}{16}}|G|^{\frac{13}{16}}$ was treated as dominated by $|G|^{\frac{17}{24}}$ throughout the relevant middle range. In fact, this domination holds only for $|G|\leq |F|^{9/5}$. The range $|F|^{9/5}<|G|<|F|^{9/4}$ therefore requires a different comparison of the available estimates, with the worst case occurring at $|G|\sim |F|^2$. This gives the corrected dual exponent $32/23$, and hence the corrected extension range $r>\frac{32}{9}$ rather than $r>\frac{24}{7}$.

This corrected result still gives an improvement over the Mockenhaupt--Tao exponent $r>\frac{18}{5}$ in arbitrary finite fields in which $-1$ is not a square, and it does so without invoking the incidence machinery used in the prime-field literature. It is, however, slightly weaker than the best currently available prime-field estimates obtained using incidence theory, $r>\frac{188}{53}$ from \cite{LewR}. 

Finally, the correction changes which term is the true obstruction. In the original version, a remark observed that the $m^{7/2}$ term in Proposition \ref{prop:TboundGF} could be improved in certain ranges, but stated that such an improvement would not affect the main theorem. That conclusion depended on the erroneous dominance calculation above. After the correction, the contribution arising from this $m^{7/2}$ term is precisely the dominant middle-range obstruction. Improvements of this term, obtained by combining the present bilinear framework with incidence estimates, will be addressed in forthcoming work.

\section{Introduction}
The Fourier restriction problem is a central problem in Euclidean harmonic analysis. The study of this phenomenon in the finite field setting was initiated by Mockenhaupt and Tao \cite{MT}, and is the subject of a large number of papers. For instance, see \cite{IKL,Le, LewR,LewK, LewN,LL,MT,RS,SD}. We refer the reader to the papers  \cite{LewK} and \cite{MT} for an overview of the problem, connections to other areas of math, and a summary of conjectures and partial progress. 

Let $F$ denote a finite field of odd characteristic, and let $F^3$ denote the three dimensional vector space over $F$. The three dimensional paraboloid is the surface $P:=\{(\underline{x},\underline{x}\cdot \underline{x}) : \underline{x} \in F^{2}\}$. We endow $P$ with the normalized counting measure $d\sigma$ which assigns $|P|^{-1}=|F|^{-2}$ to each point on $P$. For a complex-valued function on $P$ (which we naturally identify with $F^{2}$), we define the $L^q$ norm with respect to $d\sigma$ as $||f||_{L^q(P,d\sigma)} := \left(|P|^{-1} \sum_{\xi \in P} |f(\xi)|^q \right)^{1/q} $ while for a function $g: F^3 \rightarrow \mathbb{C}$ we define $||g||_{L^r(F^d)} : = \left( \sum_{x \in F^d} |g(x)|^{r} \right)^{1/r}$. For a function $f:P \rightarrow \mathbb{C}$ we define the Fourier extension operator
$$(fd\sigma)^{\vee}(x) := |F|^{-2} \sum_{\xi \in P} f(\xi) e(x\cdot \xi). $$
We will be interested in proving inequalities of the form 
\begin{equation}\label{eq:mRest} ||(fd\sigma)^{\vee}||_{L^r(F^3)} \lesssim ||f||_{L^2(P,d\sigma)}
\end{equation}
which hold for all function $f : P \rightarrow \mathbb{C}$, with the implied constant independent of $|F|$. From H\"older's inequality, lower values of $r$ imply higher values of $r$. The Stein-Tomas method establishes the inequality for $r \geq 4$ which is the optimal $L^2$ result for arbitrary finite fields. For finite fields in which $-1$ is not a square, Mockenhaupt and Tao \cite{MT} improved this range to $r > \frac{18}{5} =3.6$ and conjectured that the inequality holds for $r\geq 3$. Mockenhaupt and Tao's result was improved to $r <  \frac{18}{5} - \delta$ for some unspecified small $\delta$ by the author \cite{LewN} in arbitrary finite fields of odd order and to $r \geq \frac{745}{207}\approx 3.599$ in prime order fields. This was subsequently improved by Stevens and de Zeeuw \cite{SD} to $ r > \frac{68}{19}\approx 3.579$,  Rudnev and Shkredov \cite{RS} to $ r \geq \frac{32}{9}\approx 3.555$, and the author \cite{LewR} to $r > \frac{188}{53} \approx 3.547$. All of these improvements were obtained only in prime order fields and proceed by improving the same combinatorial input in the Mockenhaupt-Tao argument. In doing so, all of these use, indirectly, Rudnev's point-plane incidence theorem \cite{R}, which in turn relies on a highly non-trivial algebraic geometry result Koll\'ar \cite{Kol}. Rudnev and Shkredov remark in \cite{RS} that there are significant challenges with obtaining an analogous result in arbitrary finite fields. See also Rudnev's survey \cite{Rsur}.

In the current paper, we prove the following: 

\begin{theorem}\label{thm:Main}Let $F$ be a finite field of odd characteristic in which $-1$ is not a square. Then 
$$ ||(fd\sigma)^{\vee}||_{L^r(F^3)} \lesssim ||f||_{L^2(P,d\sigma)}$$
for all $r > \frac{32}{9} \approx 3.555$.
\end{theorem}

The proof proceeds by obtaining estimates for the $L^2$ norm of a related bilinear operator. In obtaining these estimates, one is lead to a new geo-combinatorial problem. Very loosely, this problem requires one to decompose a subset of $F^2$ into sub-pieces each of which either (i) do not contain a near maximal number of rectangles, or (ii) do not contain a near maximal number of trapezoids. This we are able to do with some fairly elementary arguments. See Proposition \ref{prop:TboundGF} for a precise statement of this result. Notably we avoid the use of the deep incidence estimates used in all of the prior arguments that extend below the Mockenhaupt-Tao exponent. 

Finally, we remark on the requirement in Theorem \ref{thm:Main} that $-1$ is not a square in $F$. If $-1$ is a square in $F$ then the paraboloid will contain lines. In this case the surface can be re-parameterized as a hyperbolic paraboloid $\{(x,y,xy) | (x,y) \in F^2 \}$. Testing the inequality with $f$ taken to be the characteristic function of one of these lines, one sees that the conclusion of Theorem \ref{thm:Main} can only hold in the weaker Stein-Tomas range of $r \geq 4$. This is somewhat analogous to the numerology of the three-dimensional Euclidean restriction problem (for the paraboloid or sphere). Roughly speaking, lines (or more generally affine subspaces) in the finite field setting play a role similar to `caps' in the Euclidean setting. It is conjectured that one can obtain estimates for the same range of $r$ if one replaces the $||\cdot||_{L^2(P,d\sigma)}$ norm on the right-hand side by a larger $L^q$ norm. In particular, it is conjectured that 
$$ ||(fd\sigma)^{\vee}||_{L^r(F^3)} \lesssim ||f||_{L^q(P,d\sigma)}$$
holds for $r\geq 3$ and $q \geq \frac{r}{r-2}$, with the strongest / end-point conjecture being that the inequality holds with $r=3$ with $q=3$. The author has studied this setting in \cite{LewK} and has verified the conjecture holds at, and slightly beyond, the Mockenhaupt-Tao exponent of $r \geq \frac{18}{5} -\delta$ for some small $\delta > 0$. This requires combining the Mockenhaupt-Tao machinery with substantial additional arguments to account for the potential Kakeya-like configurations of affine spaces that hypothetically could be encoded by $f$. It seems likely that the argument presented here can be combined with those arguments to make further progress in that case. We will not explore that here, however.

\section{Acknowledgment}
We would like to thank the anonymous referee for a careful reading of the manuscript.

\section{Background and Notation}
We give a brief overview of the notation and standard tools established in prior work on the problem. We refer the reader to \cite{IKL, LewK, MT} for a more detailed exposition. We will use $\mathbb{C}$ to denote the complex numbers and, unless otherwise stated, $F$ will denote a finite field of odd characteristic in which $-1$ is not a square, which we will endow with the counting measure $dx$. Given a function $f : F^3 \to \mathbb C$ we will often consider levels where the value of the function is proportional in size to a given value. In particular we write $f(x) \sim \lambda$ to indicate that $\lambda/2 \leq |f(x)| \leq 2\lambda$. Given $x \in F^3$ we will let $\underline{x}$ denote the first $2$ coordinates of $x$ and $x_3$ the $3$-th coordinate or, in other words, $x=(\underline{x},x_3)$. This allows us to parameterize the $3$-dimensional paraboloid as $P:=\{(\underline{x},\underline{x}\cdot \underline{x}) : \underline{x} \in F^{2}\}$.
By duality an estimate of the form \eqref{eq:mRest} is equivalent to the inequality 
\begin{equation}\label{eq:dual}
|| \widehat{g} ||_{L^{2}(P,d\sigma)} \lesssim ||g||_{L^{r'}(F^3)}
\end{equation}
holding for all functions $g:F^3 \rightarrow \mathbb{C}$ where $r'=\frac{r}{r-1}$ is the conjugate exponent. In this setting we have

\begin{equation}\label{eq:Convolution}||\widehat{g}||_{L^2(P,d\sigma)} = \left| \left<g, g*(d\sigma)^{\vee}\right>_{F^{3}} \right|^{\frac{1}{2}}\end{equation}
where $(d\sigma)^{\vee}$ denotes the inverse Fourier transform of the surface measure of P which an easy computation (see (18) in \cite{MT}) gives as
\begin{equation}\label{eq:evaulation}(d\sigma)^{\vee}(x) = |F|^{-2} \sum_{\xi \in P} e(x\cdot \xi) 
= \left\{\begin{array}{ll} \sigma_F |F|^{-1}  e\left(\frac{\underline{x}\cdot\underline{x}}{-4x_3}\right)\quad &\mbox{if} \quad x_3\ne 0\\

 0 \quad &\mbox{if} \quad x_3=0,~x\ne 0 \\
                1 \quad &\mbox{if} \quad x=0,\end{array}\right. 
\end{equation}
where $\sigma_F$ is a complex number such that $|\sigma_F|=1$ and $\chi(\cdot)$ denotes the multiplicative quadratic character. It follows that $|(d\sigma)^{\vee}(x)| \leq |F|^{-1}$ for $x\ne 0.$ For notational convenience we define $K(x):F^{3} \rightarrow \mathbb{C}$ by $(d\sigma)^{\vee} = \delta(x) + K(x)$
where $\delta(x)$ denotes the Dirac delta function. With this notation, one has 
\begin{equation}\label{eq:decayL1}||\widehat{g}||_{L^2(P,d\sigma)} = \left| \left<g, g*(d\sigma)^{\vee}\right> \right|^{\frac{1}{2}} \leq ||g||_{L^2(F^3)} + |F|^{-\frac{1}{2}}||g||_{L^1(F^3)}. \end{equation}

Interpolating this estimate with the following consequence of Parseval 
\begin{equation}\label{eq:bigL2}
\|\widehat{g}\|_{L^2(P, d\sigma)} \le |F|^{\frac{1}{2}} ||g||_{L^{2}(F^3)}
 \end{equation} yields the Stein-Tomas estimate
$$ ||\widehat{g}||_{L^{2}(P,d\sigma)} \lesssim ||g||_{L^{\frac{4}{3}}(F^3)} .$$
We refer the reader to \cite{MT} for additional detail. Another well-known consequence of \eqref{eq:decayL1} is the following epsilon removal lemma (see \cite{LewK} Lemma 16) which states:
\begin{lemma}\label{lem:epsRem}($\epsilon$-removal) Assume that for every $\epsilon>0$ one has a constant $c(\epsilon)$ such that the following inequality holds
$$ \|\widehat{g}\|_{L^2(P, d\sigma)} \lesssim c(\epsilon) |F|^{\epsilon} \|g\|_{L^{r_{0}}(F^3)}$$
then for every $r>r_0$, we have the estimate
$$ \|\widehat{g}\|_{L^2(P, d\sigma)} \lesssim_r \|g\|_{L^{r}(F^3)}.$$
\end{lemma}
We also recall that dyadic pigeonholing allows us to restrict attention to functions $g$ that are approximately constant on a set when proving inequalities of the form \eqref{eq:dual}. In particular:
\begin{lemma}\label{lem:dyPig}(Dyadic Pigeonholing) If one proves the inequality 
$$ ||\widehat{g}||_{L^2(P, d\sigma)} \lesssim c(\epsilon) |F|^{\epsilon} \|g\|_{L^{r_{0}}(F^3)}$$
for all $g: F^3 \rightarrow \mathbb{C}$ of the form $g(x) \sim 1$ on its support $G \subseteq F^3$, then for every $r>r_0$, the inequality
$$ \|\widehat{g}\|_{L^2(P, d\sigma)} \lesssim \|g\|_{L^{r}(F^3)}$$
holds for arbitrary $g:F^3 \rightarrow \mathbb{C}$.
\end{lemma}

\begin{proof}
We perform a dyadic pigeonholing to reduce to functions that are approximately characteristic functions of sets. Normalize $g$ such that $||g||_{L^{r_{0}}} =1$, and write $g= \sum_{i=0}^{10 \log |F|} g 1_{G_i} + g 1_{G_{\text{small}}}$ where $G_i :=\{x\in F^{3}: f(x) \sim 2^{i} \}$ and $G_{\text{small}}:=\{x \in F^{3}: g(x) \leq |F|^{-10} \}$.  Trivially, we have  $ || \widehat{g 1_{G_{\text{small}}}} ||_{L^{2}(P,d\sigma)} \leq O(1)$, thus
$$|| \widehat{g} ||_{L^{2}(P,d\sigma)} \leq O(1)+  \sum_{i=0}^{10 \log |F|} || \widehat{g 1_{G_i}} ||_{L^{2}} \lesssim O(1) + \log |F| \times \max_{i}  2^{-i} || 2^{i}\widehat{g 1_{G_i}} ||_{L^{2}}.$$
Now $2^{i} g 1_{G_i} \sim 1$ on its support $G_i$. By hypothesis, we have that $||  2^{i}\widehat{g 1_{G_i}} ||_{L^{2}(P,d\sigma)} \lesssim c(\epsilon) |F|^{\epsilon} |G_i|^{1/r_0}$, for every $\epsilon >0$. On the other hand, the normalization $||g||_{L^{r_0}} =1$ implies $2^{-i}|G_i|^{1/r_0} \lesssim 1$. Putting this together we have
$$|| \widehat{g} ||_{L^{2}(P,d\sigma)} \lesssim c(\epsilon) |F|^{\epsilon}\|g\|_{L^{r_{0}}(F^3)}$$ 
and the claim follows from Lemma \ref{lem:epsRem}.
\end{proof}

Next we recall some combinatorial quantities that will play a role in the proof. This requires we introduce some notation. Let $A \subset F^2$ we say that $(x_0,x_1,x_2)$ (with $ x_0,x_1,x_2 \in F^2$) forms a corner if
$$(x_1-x_0)\cdot(x_2-x_1)=0 $$
and that $(x_0,x_1,x_2,x_3)$ is a rectangle if each triple $(x_i,x_{i+1},x_{i+2},x_{i+3})$ forms a corner for $i\in{0,1,2,3}$ and with the indexing arithmetic modulo $4$. For a given rectangle, we will call the four pairs of points of the form $(x_i,x_{i+1})$ adjacent points. For $A \subseteq F^2$, we define $\mathcal{R}(A)$ to be the number of rectangles in $A$. Note that if we replace $F^2$ with the Euclidean plane, these definitions correspond to the natural geometric interpretation of these terms. An elementary combinatorial argument using the Cauchy-Schwarz inequality shows that 
\begin{equation}\label{eq:Rbound}\mathcal{R}(A) \lesssim |A|^{5/2}
\end{equation}
for all subsets $A \subseteq F^2.$ This estimate is implicit in the original paper of Mockenhaupt-Tao \cite{MT}, however is not formulated in this manner there. See \cite{RS} and \cite{LewR}. For large sets, an incidence of estimate of Vihn \cite{V} implies (see (15) in \cite{RS}) the bound
\begin{equation}\label{eq:VRbound}\mathcal{R}(A) \lesssim |F|^{-1}|A|^3 + |A|^2 |F|^{1/2}.
\end{equation}
This improves on \eqref{eq:Rbound} in the range $|F|\leq|A| \leq |F|^2$. For finite fields of prime order $p$, and $A$ sufficiently small, this estimate can be improved using recent advances in incidence theory (related to sum-product estimates). Using ideas from \cite{SD} and \cite{RS}, in \cite{LewR} the following was proved:
\begin{theorem}\label{thm:bestRec}Let $F_p$ denote a finite field of prime order $p$ and in which $-1$ is not a square. Let $A \subset F^2$ such that $|A|\leq p^\frac{26}{21}$ then
$$\mathcal{R}(A) \lesssim_{\epsilon}|A|^{\frac{99}{41}+\epsilon} $$
for any $\epsilon >0$.
\end{theorem}
The conjectured optimal form of the Szemer\'edi-Trotter theorem in fields of prime order would imply this result with exponent $2+\epsilon$. 

Next we recall the Mockenhaupt-Tao machine, as formulated in \cite{IKL} (see also \cite{RS}). Let $g:F^3 \rightarrow \mathbb{C}$ such that $g \sim 1$ on its support $G \subseteq F^3$. We define $g_z : F^{2} \rightarrow \mathbb{C}$ by $g_z(x)= g(x,z)$ with support $G_z \subseteq F^2$.  With this notation we have the following

\begin{lemma}\label{lem:MT0} Let $g : F^3 \rightarrow \mathbb{C}$ such that $g \sim 1$ on its support $G$. We then have
$$ \|\widehat{g}\|_{L^2(P, d\sigma)} \lesssim |G|^{\frac{1}{2}} + |G|^{\frac{3}{8}} |F|^{-\frac{1}{8}}\left( \sum_{z \in F} (\mathcal{R}(G_z))^{1/4} \right)^{\frac{1}{2}} .$$
\end{lemma}
Using \eqref{eq:Rbound} and the fact that $\sum_{z}|G_z|^{5/8} \leq |F|^{\frac{3}{8}} |G|^{5/8}$ the above implies
\begin{equation}\label{eq:MTwCS}\|\widehat{g}\|_{L^2(P, d\sigma)} \lesssim  |G|^{\frac{1}{2}} + |G|^{\frac{3}{8}} |F|^{-\frac{1}{8}} \left( \sum_{z\in F} |G_z|^{\frac{5}{8}} \right)^{\frac{1}{2}} \lesssim  |G|^{\frac{1}{2}} + |F|^{\frac{1}{16}}  |G|^{\frac{11}{16}} 
\end{equation}
Combining the various estimates we have present we have the following:
\begin{lemma}Let $F$ be an arbitrary finite field of odd order in which $-1$ is not a square. Let $g: F^3 \rightarrow \mathbb{C}$ such that $g \sim 1$ on its support $G \subseteq F^3$. Then one has the following estimates:

\begin{equation}\label{eq:UcaseEst}|| \hat{g}||_{L^{2}(P,d\sigma)} \lesssim |G|^{\frac{1}{2}}+ \left\{\begin{array}{ll} |G| |F|^{-\frac{1}{2}}\quad &\mbox{if } |G| \leq |F|^{\frac{9}{5}}\\
|G|^{\frac{11}{16}} |F|^{\frac{1}{16}} &\mbox{if } |F|^{\frac{9}{5}} \leq  |G| \leq |F|^{2} \\
|G|^{\frac{5}{8}} |F|^{\frac{3}{16}} &\mbox{if } |F|^{2} \leq  |G| \leq |F|^{\frac{5}{2}} \\

|G|^{\frac{1}{2}} |F|^{\frac{1}{2}} &\mbox{if } |F|^{\frac{5}{2}} \leq  |G| \leq |F|^{3} \end{array}\right. 
\end{equation}
\end{lemma}
\begin{proof}The first estimate follows from \eqref{eq:decayL1} and the last estimate follows from \eqref{eq:bigL2}. The second estimate can be obtained from using \eqref{eq:Rbound} in Lemma \ref{lem:MT0}. In particular this gives
$$\|\widehat{g}\|_{L^2(P, d\sigma)} \lesssim  |G|^{\frac{1}{2}} + |G|^{\frac{3}{8}} |F|^{-\frac{1}{8}} \left( \sum_{z\in F} |G_z|^{\frac{5}{8}} \right)^{\frac{1}{2}} \lesssim  |G|^{\frac{1}{2}} + |F|^{\frac{1}{16}}  |G|^{\frac{11}{16}}. $$
The third estimate is obtained from using \eqref{eq:VRbound} in Lemma \ref{lem:MT0}, as follows:
$$\|\widehat{g}\|_{L^2(P, d\sigma)} \lesssim  |G|^{\frac{1}{2}} + |G|^{\frac{3}{8}} |F|^{-\frac{1}{8}} \left(  |F|^{-\frac{1}{8}}\left( \sum_{z\in F} |G_z|^{\frac{3}{4}} \right)^{\frac{1}{2}}  + |F|^{\frac{1}{16}}\left( \sum_{z\in F} |G_z|^{\frac{1}{2}} \right)^{\frac{1}{2}}    \right)$$
$$\lesssim  |G|^{\frac{3}{4}} |F|^{-\frac{1}{8}}+ |G|^{\frac{5}{8}}|F|^{\frac{3}{16}} \lesssim |G|^{\frac{5}{8}}|F|^{\frac{3}{16}}.$$ \end{proof}
One can check that the Mockenhaupt-Tao estimate $\|\widehat{g}\|_{L^2(P, d\sigma)} \lesssim ||g||_{L^\frac{18}{13}(F^3)}$ follows\footnote{One losses the endpoint of the inequality in passing from level sets to arbitrary functions via the $\epsilon$ removal lemma. However, with more care, on can recover the endpoint, see \cite{LL}.} by applying each of these estimates for the given range of $|G|$. 

In finite fields of prime order on can use the stronger estimate of Theorem \ref{thm:bestRec} to improve these inequalities. The following estimates can be extracted from the argument in \cite{LewR}. We omit the details as we will not use this result, and these inequalities will be superseded by the estimates we obtain below.

\begin{lemma}Let $F$ be a prime order field in which $-1$ is not a square. Let $g: F^3 \rightarrow \mathbb{C}$ such that $g \sim 1$ on its support $G \subseteq F^3$. Then one has the following estimates:

\begin{equation}\label{eq:UcaseEstPrime}|| \hat{g}||_{L^{2}(P,d\sigma)} \lesssim_{\epsilon} |G|^{\frac{1}{2}}+ \left\{\begin{array}{ll} |G| |F|^{-\frac{1}{2}}\quad &\mbox{if } |G| \leq |F|^{\frac{94}{53}}\\
|G|^{\frac{111}{164} +\epsilon}  |F|^{\frac{3}{41}} &\mbox{if } |F|^{\frac{94}{53}} \leq  |G| \leq |F|^{\frac{75}{34}} \\
|G|^{\frac{5}{8}} |F|^{\frac{3}{16}} &\mbox{if } |F|^{\frac{75}{34}} \leq  |G| \leq |F|^{\frac{5}{2}} \\

|G|^{\frac{1}{2}} |F|^{\frac{1}{2}} &\mbox{if } |F|^{\frac{5}{2}} \leq  |G| \leq |F|^{3} \end{array}\right. 
\end{equation}

\end{lemma}

One can check that these estimates imply $\|\widehat{g}\|_{L^2(P, d\sigma)} \lesssim ||g||_{L^{\frac{188}{135} - \epsilon}(F^3)}$, which is the main result from \cite{LewR}.

The remainder of this paper will be devoted to developing some improved estimates when $|G|$ in certain ranges.

\section{A Bilinear Estimate}
In this section we develop a bilinear version of Lemma \ref{lem:MT0}. As before, we will let $g : F^3 \rightarrow \mathbb{C}$ such that $g \sim 1$ on its support $G \subseteq F^3$. For $z \in F$ we consider the function obtained by restricting G to the vertical slice at height $z$, $g_z(x) := g(\underline{x})\delta(x_3-z)$. We let  $G_z$ denote the support of $g_z$. Now one has
\begin{equation}\label{eq:start} ||\hat{g}||_{L^{2}(P,d\sigma)} = | \left<g, g*(d\sigma)^{\vee}\right> |^{1/2} \lesssim ||g||_{L^2(F^3)} +  ||g||_{L^{4/3}(F^3)}^{1/2} ||g * K||_{L^4(F^3)}^{1/2} .\end{equation}
Expanding:
$$ ||g * K||_{L^4(F^d)}= || \sum_{z \in F} g_z * K ||_{L^4(F^3)} = || \sum_{z \in F} g_z* K \sum_{z' \in F} g_{z'} *K ||_{L^2(F^3)}^{1/2} $$
we have that
\begin{equation}\label{eq:BiMT} ||\hat{g}||_{L^{2}(P,d\sigma)} \lesssim  ||\hat{g}||_{L^2(F^3)} + |G|^{\frac{3}{8}} \left(\sum_{z \in F}   || g_z* K ||_{L^4(F^3)}^2 + \sum_{\substack{ z, z' \in F  \\ z\neq z'} } || g_z* K \times g_{z'} *K ||_{L^2(F^3)}  \right)^{1/4} .\end{equation}

The key new input will be an estimate for $|| g_z* K \times g_{z'} *K ||_{L^2(F^3)}$ with $z\neq z'$. First we note that linear estimates imply bilinear estimates by Cauchy-Schwarz:
\begin{equation}\label{eq:LinToBil}|| g_z* K \times g_{z'} *K ||_{L^2(F^3)} \lesssim ||g_z* K||_{L^4(F^3)} .||g_{z'}* K||_{L^4(F^3)}\end{equation}
This holds even without the restriction that $z \neq z'$. Also recall\footnote{This is well-known and has been used in prior work on the problem. However we could not find a clear statement in this form to cite, so we have included a short proof in the appendix (see Lemma \ref{lem:Rrel}) for self-containedness. } that the $L^4$ norm of the linear operator is controlled by the combinatorial quantity $\mathcal{R}$, introduced above:
\begin{equation}\label{eq:gzToR}||g_z* K||_{L^4(F^3)}^4 \lesssim  |F|^{-1} \mathcal{R}(G_z).
\end{equation}
Using this we have
\begin{equation}\label{eq:LinToBilR}|| g_z* K \times g_{z'} *K ||_{L^2(F^3)} \lesssim |F|^{-\frac{1}{2}} \left(\mathcal{R}(G_z) \mathcal{R}(G_{z'}) \right)^{1/4}. \end{equation}

Inserting \eqref{eq:LinToBilR} into \eqref{eq:BiMT} recovers Lemma \ref{lem:MT0}. Our goal is to obtain a new estimate for the quantity $|| g_z* K \times g_{z'} *K ||_{L^2(F^3)}$ when $z \neq z'$.

First we need some notation. Let $x_1,x_2,x_3,x_4 \in F^2$. We will say that $(x_1,x_2,x_3,x_4)$ forms a trapezoid if there exists a $\lambda \in F$ such that  $(x_1-x_2)= \lambda (x_3-x_4)$. We define $\mathcal{T}(A)$ to be the number of trapezoids in $A$. Similarly, for $A,B \subseteq F^2$ we define $\mathcal{T}(A,B)$ to be the number of trapezoids $(x_1,x_2,x_3,x_4)$ with $x_1,x_2 \in A$ and $x_3,x_4 \in B$. The bilinear operator is related to $\mathcal{T}(\cdot, \cdot)$ by the following:

\begin{proposition}\label{prop:bilinear} Let $z, z' \in F$ with $z \neq z'$. Let $g_z, g_{z'}: F^3 \rightarrow \mathbb{C}$ be functions supported on the vertical hyperplanes, as defined above, such that $g \sim 1$ on their supports $G_z, G_{z'} \subseteq F^2$. We have that 
$$|| g_z* K \times g_{z'} *K ||_{L^2(F^3)}^2 \lesssim |F|^{-1} |G_z| |G_{z'}| + |F|^{-2}\mathcal{T}(G_z, G_{z'})$$
\end{proposition}

\begin{proof} Let $s_1=z$ and $s_2=z'$, and let $f_1,f_2 : F^{2} \rightarrow \mathbb C $ be defined by $f_1(x)=g_z(x,s_1)$ and $f_2(x)=g_{z'}(x,s_2)$. We may then compute $$|| g_z* K \times g_{z'} *K ||_{L^2(F^d)}^2$$
$$=  |F|^{-4} \sum_{\substack{y \in F^{2}; t \in F  \\ t \neq s_1,s_2 }} \left| \sum_{x_1 \in F^2}f_1(x_1) e\left(-\frac{(x_1-y)\cdot(x_1 -y) }{4(t-s_1)}   \right) \sum_{x_2 \in F^2}f_2(x_2) e\left(-\frac{(x_2-y)\cdot(x_2 -y) }{4(t-s_2)} \right)  \right|^2 $$
Removing factors of modulus one, thanks to the absolute values, and translating $t$ by $s_1$ and writing $s=s_2-s_1$, we have:

$$=  |F|^{-4} \sum_{\substack{y \in F^{2}; t \in F  \\ t \neq 0,s }} \left| \sum_{x_1 \in F^2}f_1(x_1) e\left(-\frac{x_1\cdot x_1 -2 x_1\cdot y }{4t}   \right) \sum_{x_2 \in F^2}f_2(x_2) e\left(-\frac{x_2\cdot x_2 -2 x_2\cdot y }{4(t-s)} \right)  \right|^2. $$

Expanding the square and reordering the sums we have that this

$$\sum_{x_1,x_2,x_3,x_4 \in F^2} f_1(x_1)\overline{f_1(x_3)}f_2(x_2)\overline{f_2(x_4)} $$
$$\times \sum_{\substack{y \in F^{2}; t \in F  \\ t \neq 0,s }} e \left( \frac{x_3\cdot x_3 -x_1 \cdot x_1 +2x_1 \cdot y - 2x_3 \cdot y }{4t} + \frac{x_4\cdot x_4 -x_2 \cdot x_2 +2x_2 \cdot y - 2x_4 \cdot y  }{4(t-s)} \right). $$
We may rewrite inner sum as:
$$\sum_{\substack{t \in F  \\ t \neq 0,s }} e \left( \frac{x_3\cdot x_3 -x_1 \cdot x_1 }{4t} + \frac{x_4\cdot x_4 -x_2 \cdot x_2 }{4(t-s)} \right) \sum_{y\in F^2} e \left( y\cdot \left( \frac{x_1 - x_3}{2t} + \frac{x_2  - x_4   }{2(t-s)} \right) \right).$$

The inner sum will vanish unless 
\begin{equation}\label{eq:x1}x_1-x_3 = \frac{t(x_2-x_4)}{t-s}.\end{equation}
If $x_1 \neq x_3$ or $x_2 \neq x_4$, this will occur for at most one value of $t$ if and only if $x_1-x_3$ is a scalar multiple of $x_2-x_4$ or, in other words, $(x_1,x_3,x_2,x_4)$ form a trapezoid.
Taking into account the contribution also from the quadruples with $x_1=x_3$ and $x_2=x_4$, this shows that the quantity above is at most
$$ |F|^{-1} |G_z| |G_{z'}| + |F|^{-2}\mathcal{T}(G_z,G_{z'})$$
which completes the proof.
\end{proof}

Let us pause to explain why the reader familiar with the numerology of the problem might be skeptical of the usefulness of Proposition \ref{prop:bilinear}. Consider $|G_z|\sim |G_{z'}|\sim m$ for some fixed $m$. The weakest of the estimates we have given for $\mathcal{R}(\cdot)$, ($\mathcal{R}(A) \lesssim |A|^{\frac{5}{2}}$) implies, via \eqref{eq:LinToBil}, that $|| g_z* K \times g_{z'} *K ||_{L^2(F^3)} \lesssim |F|^{-\frac{1}{2}} m^{\frac{5}{4}} $. On the other hand, since $\mathcal{T}(G_z, G_{z'})$ can be as large as $m^{4}$, our bilinear estimate only gives $|| g_z* K \times g_{z'} *K ||_{L^2(F^3)} \lesssim |F|^{-1} m^2$, which is inferior in the range of $m$ we must control. Indeed this can happen if $G_z$ and $G_{z'}$ are concentrated on a small number of parallel lines. In other words, in this case our new estimate gives a worse bound than the previously available estimates.

Our key observation is that the sets that give near-maximal values of $\mathcal{T}(A, B)$ are, in some sense, not compatible with the sets that give near maximal values of $\mathcal{R}(A) \mathcal{R}(B)$. Defining
$$\mathcal{B}(A,B) := \min\{\mathcal{T}(A, B), |F| \left( \mathcal{R}(A) \mathcal{R}(B)\right)^{1/2} \}$$
combining the linear and bilinear estimate we have:
\begin{lemma}\label{lem:biCont} In the notation above we have:
$$|| g_z* K \times g_{z'} *K ||_{L^2(F^3)}^2 \lesssim |G_z||G_{z'}| |F|^{-1} +  |F|^{-2} \mathcal{B}(G_{z},G_{z'}).$$
\end{lemma}

In the next section we will show that geometric considerations imply better estimates on $\mathcal{B}(A,B)$ than are implied by the best known estimates for $\mathcal{R}(\cdot)$.

\section{A Geometric Estimate}
Our aim will be to obtain various estimates on the quantity $\mathcal{T}$ introduced in the last section. We will always assume that $F$ is a finite field of odd characteristic in which $-1$ is not a square. Trivially
$$\mathcal{T}(A,B) \leq |A|^2|B|^2, $$
which, as we have already remarked, is nearly sharp in the sense that $\mathcal{T}(A,B)\sim |A|^2|B|^2$ if all of the points lie on a small number of parallel lines. Unlike in the case of $\mathcal{R}$ one can not hope for better estimates simply by restricting the field or the cardinalities of $A$ and $B$. Our observation is that this estimate for $\mathcal{T}(A,B)$ is only extremized when the sets $A$ and $B$ concentrate on a few lines, while the known estimates for $\mathcal{R}(\cdot)$ are only extremized when the set concentrates on a grid with more evenly balanced side lengths/projections. We make this more precise in the following results.

To avoid some technical complications, we will work with what we call $k$ regular sets. We say that a set $A \subset F^2$ is $k$ regular (for $k \leq |A|^{\frac{1}{2}}$) if it can be written as the disjoint union of points from $\sim k$ distinct lines each containing $\sim |A|k^{-1}$ points. We call the $k$ lines used in the definition the frame of $A$, and note that if $\ell$ is not in the frame of $A$ then $|\ell \cap A| \leq k$. We will call a set irregular if $|\ell \cap A| \leq |A|^\frac{1}{2} $ for all $\ell$. The usefulness of this definition follows from the fact we can always decompose a set as the disjoint union of $O( \log |F| )$ $k$-regular subsets and an irregular set.

\begin{lemma}\label{lem:kregPart}A set $A \subseteq F^2$ can be decomposed as a disjoint union of 
$$ A = \bigcup_{\substack{k \text{ dyadic}  \\ 1 \leq k \leq |A|^{1/2} }} A_k + A_*$$
where $A_k$ is $k$ regular and $A_*$ is irregular.
\end{lemma}
\begin{proof}The decomposition is constructed with a greedy algorithm as follows. Let $k_{0}$ be the maximal size of the intersection of a line with $A$. We iteratively build $A_k$, $k$ dyadic, by selecting the points in $A$ that lie on a line with maximal intersection with $A$, say, $k'$ and then removing those points from $A$ and adding them to $A_k$ where $k' \sim k$, until $|A \cap \ell | \leq |A|^{1/2}$ and then we set this set to be $A_*$.
\end{proof}

We now turn to estimating $\mathcal{T}(\cdot,\cdot)$.
\begin{lemma}\label{lem:Test}Let $A_1,A_2 \subseteq F^2$ be sets each formed as the union of $|A_i|k_i^{-1}$ points from $\sim  k_i$ lines, where $k_i \leq |A_i|^{\frac{1}{2}}$.  Then 
$$|\mathcal{T}(A_1,A_2)| \lesssim  |A_1|^2|A_2| k_2 + |A_1||A_2|^2 k_1 +|A_1|^2 |A_2|^2 k_1^{-1}k_2^{-1}. $$
Furthermore if $A_1,A_2$ are irregular, then
$$|\mathcal{T}(A_1,A_2)| \lesssim \min\{ |A_1|^{\frac{3}{2}} |A_2|^{2}, |A_2|^{\frac{3}{2}} |A_1|^{2}\}.$$
\end{lemma}
\begin{proof} Let $L_i$ denote the frame of $A_i$. We will make use of the fact that if $\ell \notin L_i$  then $|\ell \cap A_i| \lesssim k_i$. We wish to count the number of trapezoids $(x_1,x_2,y_1,y_2)$ with $x_1,x_2 \in A_1$ and $y_1,y_2 \in A_2$. We first consider the contribution from quadruples where the lines through $(x_1,x_2)$ and $(y_1,y_2)$ are in $L_1$ and $L_2$, respectively. This can be bounded by 
$$|L_1|^2 |L_2|^2 k_1 k_2  \sim |A_1|^2 k_1^{-2} |A_2|^2 k_2^{-2}  k_1 k_2  \sim |A_1|^2 |A_2|^2 k_1^{-1}k_2^{-1} $$
which is the last term in the statement of the lemma. Next we consider the contribution from trapezoids where the line through $(y_1,y_2)$ is not in $L_2$. The number of such quadruples is at most the number of triples $(x_1,x_2,y_3)$ multiplied by the maximum number of points on the line through $y_3$ parallel to $x_1,x_2$. As we remarked at the start, this latter quantity is is at most $k_2$. Thus this contribution is at most
$|A_1|^2 |A_2| k_2.$
Switching the roles of $A_1$ and $A_2$, we see the contribution of quadruples such that the line through $(x_1,x_2)$ is not in $L_1$ is similarly bounded by 
$|A_1| |A_2|^2 k_1.$
Putting these estimates together completes the proof of the first estimate. The second estimate follows from this latter argument with $k_i \sim |A_i|^{\frac{1}{2}}$.
\end{proof}

\begin{lemma}\label{lem:Rest}Let $A \subseteq F^2$ be a $k$ regular set. Then 
$$\mathcal{R}(A)\lesssim|A|^2 k .$$
\end{lemma}
\begin{proof}We let $L$ denote the frame of $A$. We first count the number of rectangles with sides in $L$. Let us consider $(x_0,x_1,x_2,x_3) \in A^4$ that form a rectangle. Consider a corner which we will write, without loss of generality, as $$(x_1-x_0)\cdot(x_2-x_1)=0. $$ Let us write $\ell_1$ for the line through $(x_1,x_0)$ and $\ell_2$ for the line through $(x_2,x_1)$. 

First consider the case where  $\ell_1, \ell_2 \in L$. The contribution of such rectangles can be bounded\footnote{We implicitly use here that $-1$ is not a square in $F$. Without this restriction $F$ will contain isotropic lines, and any three points on an isotropic line will form a corner.} by the number of possible pairs $(\ell_1, \ell_2)$ times the number of pairs of points with one on each of these lines, or 
$$|L|^{2} \times |A \cap \ell_1 | \times |A \cap \ell_2 | \lesssim k^2 \times |A|^2 k^{-2} \lesssim |A|^2. $$
Next we consider the contribution of rectangles where $\ell_1 \in L$ and $\ell_2 \notin L$. Here we will use  that $|\ell_2 \cap A |\lesssim k$. We can bound the contribution from such corners/rectangles as the number of ways of selecting $\ell_1$ times the square of the number of points on $\ell_1$ times the number of points on $\ell_2$, or
$$k \times |A|^2 k^{-2} \times k \lesssim |A|^2.$$

It now suffices to bound the contribution from rectangles where the line through none of the adjacent points $(x_i,x_{i+1})$ are in $L$. Thus we may assume that each line through a side of the rectangle intersects $A$ in at most $k$ points. We can upper bound the number of such corners by selecting any two points in $A$ for $x_0$ and $x_1$ in the relation $(x_1-x_0)\cdot(x_2-x_1)=0 $, and noting that the possible $x_2$ must lie on a line $\ell \notin L$, and thus can contain at most $k$ points. Thus the number of such corners is at most $|A|^2 k$, which completes the proof.
\end{proof}

We now combine the prior two estimates:

\begin{proposition}\label{prop:TboundGF} Let $A, B \subseteq F^2$ be $k$ regular sets (or both irregular sets) with $|A|\sim |B| \sim m$. Then
$$\min \{ |F| \left(\mathcal{R}(A)\mathcal{R}(B)\right)^{1/2}, \mathcal{T}(A,B) \} \lesssim m^{8/3}|F|^{2/3} +m^{7/2}.$$
\end{proposition}
\begin{proof}Let $k_* = m^{2/3} |F|^{-1/3}$. If $k \leq k_*$ then the result immediately follows from Lemma \ref{lem:Rest}. If $k \geq k_{*}$ (or the sets are irregular) we apply Lemma \ref{lem:Test} which gives
$$\mathcal{T}(A,B) \lesssim \max_{m^{2/3} |F|^{-1/3} \leq k \leq m^{1/2}} \left( m^{3}k + m^4 k^{-2} \right) \lesssim  m^{7/2} + m^{8/3}|F|^{2/3}.$$  
The result follows.
\end{proof}

\section{Proof of Theorem \ref{thm:Main}}
We will say that a function $g: F^3 \rightarrow \mathbb{C}$ is ``regular" if (1) $g \sim 1$ on its support, (2) there is a constant $m$ such that the support of each non-empty vertical slice of $G$, say $G_z$, has size $\sim m$, and (3) either there is a constant $k$ such that each nonempty slice $G_z$ is $k$-regular or each nonempty slices is irregular.

From Lemma \ref{lem:kregPart} and dyadic pigeonholing (see Lemma \ref{lem:dyPig}) an arbitrary function $g: F^3 \rightarrow \mathbb{C}$ can be written as the sum of (scalar multiples of) $\log^{O(1)} |F|$ regular functions with disjoint supports. Thus it suffices to prove the inequality \eqref{eq:dual}, dual to that in Theorem \ref{thm:Main}, for regular functions. We will think of $m$, the size of each nonempty slices of $G$, as fixed throughout the argument. We will also work with the number of nonempty slices, which we will call $w$. Clearly the $m$ and $w$ are related as $m \sim |G| w^{-1}$.

Using \eqref{eq:BiMT}  and \eqref{eq:gzToR} with Lemma \ref{lem:biCont}, we have that
$$ ||\hat{g}||_{L^{2}(P,d\sigma)} \lesssim  ||g||_{L^2(F^3)} + |G|^{\frac{3}{8}} \left(\sum_{z \in F}   || g_z* K ||_{L^4(F^3)}^2 + \sum_{\substack{ z, z' \in F  \\ z\neq z'} } || g_z* K \times g_{z'} *K ||_{L^2(F^3)}  \right)^{1/4}$$
$$\lesssim   |G|^{\frac{1}{2}} + |G|^{\frac{3}{8}} |F|^{-\frac{1}{8}} \left(\sum_{z \in F}  (\mathcal{R}(G_z))^{1/2} + |F|^{-\frac{1}{2}} \sum_{\substack{ z, z' \in F  \\ z\neq z'} } (\mathcal{B}(G_z,G_{z'}))^{1/2} +
 \sum_{\substack{ z, z' \in F  \\ z\neq z'} } |G_z|^\frac{1}{2} |G_{z'}|^\frac{1}{2}  \right)^{1/4}.  $$  
Using Proposition \ref{prop:TboundGF} (and, say, \eqref{eq:Rbound} to bound the first term) we may bound the above by:

$$ |G|^{\frac{1}{2}} + |G|^{\frac{3}{8}}|F|^{-\frac{1}{8}} \left( \sum_{z \in F} |G_z|^{\frac{5}{4}} + |F|^{-\frac{1}{2}} w^2 m^{\frac{7}{4}} + |F|^{-\frac{1}{2}} w^2 |F|^{\frac{1}{3}} m^{\frac{4}{3}}  + |F| |G|   \right)^{1/4}.     $$
Using $m=|G|w^{-1}$

$$ |G|^{\frac{1}{2}} + |G|^{\frac{3}{8}}|F|^{-\frac{1}{8}} \left( \sum_{z \in F} |G_z|^{\frac{5}{4}} + |F|^{-\frac{1}{2}} w^2 w^{-\frac{7}{4}}|G|^{\frac{7}{4}} + |F|^{-\frac{1}{2}} w^2 |F|^{\frac{1}{3}} w^{-\frac{4}{3}} |G|^{\frac{4}{3}}   + |F| |G|  \right)^{1/4}     $$

$$ \lesssim |G|^{\frac{1}{2}} + |G|^{\frac{3}{8}}|F|^{-\frac{1}{8}} \left( \sum_{z \in F} |G_z|^{\frac{5}{4}} + |F|^{-\frac{1}{2}} w^{\frac{1}{4}} |G|^{\frac{7}{4}} + |F|^{-\frac{1}{6}} w^{\frac{2}{3}} |G|^{\frac{4}{3}}  + |F| |G|   \right)^{1/4}     $$

using that $w \leq |F|$ we have that this is 

$$ \lesssim |G|^{\frac{1}{2}} + |G|^{\frac{3}{8}}|F|^{-\frac{1}{8}} \left( |G|^{\frac{5}{16}} + |F|^{-\frac{1}{16}}  |G|^{\frac{7}{16}} + |F|^{\frac{1}{8}}  |G|^{\frac{1}{3}}  + |F|^{\frac{1}{4}} |G|^{\frac{1}{4}}   \right)     $$

$$\lesssim  |G|^{\frac{1}{2}} +  |G|^{\frac{11}{16}} |F|^{-\frac{1}{8}} + |F|^{-\frac{3}{16}}  |G|^{\frac{13}{16}} +  |G|^{\frac{17}{24}}    +|F|^{\frac{1}{8}} |G|^{\frac{5}{8}}.        $$

Combining this estimate with the earlier estimates in \eqref{eq:UcaseEst}, and keeping the largest remaining term in each range, we obtain the following corrected table.

\begin{lemma}\label{lem:correctedFinal}Let $F$ be a finite field of odd order in which $-1$ is not a square. Let $g: F^3 \rightarrow \mathbb{C}$ be a regular function such that $g \sim 1$ on its support $G \subseteq F^3$. Then one has the following estimates:

\begin{equation}\label{eq:UcaseEstN}|| \hat{g}||_{L^{2}(P,d\sigma)} \lesssim |G|^{\frac{1}{2}}+ \left\{\begin{array}{ll} |G| |F|^{-\frac{1}{2}}\quad &\mbox{if } |G| \leq |F|^{\frac{12}{7}}\\
|G|^{\frac{17}{24}} &\mbox{if } |F|^{\frac{12}{7}} \leq |G| \leq |F|^{\frac{9}{5}}\\
 |G|^{\frac{13}{16}} |F|^{-\frac{3}{16}}  &\mbox{if } |F|^{\frac{9}{5}} \leq  |G| \leq |F|^{2} \\
|G|^{\frac{5}{8}} |F|^{\frac{3}{16}} &\mbox{if } |F|^{2}\leq  |G| \leq |F|^{\frac{5}{2}} \\

|G|^{\frac{1}{2}} |F|^{\frac{1}{2}} &\mbox{if } |F|^{\frac{5}{2}} \leq  |G| \leq |F|^{3}. \end{array}\right. 
\end{equation}
\end{lemma}

Indeed, the only point that requires care is the middle range: the term $|F|^{-\frac{3}{16}}|G|^{\frac{13}{16}}$ is dominated by $|G|^{\frac{17}{24}}$ only when $|G|\leq |F|^{\frac95}$, and remains the dominant term up to $|G|\leq |F|^2$. One checks from \eqref{eq:UcaseEstN} that every term is bounded by $|G|^{\frac{23}{32}}$. Thus
 $$|| \hat{g}||_{L^{2}(P,d\sigma)} \lesssim |G|^{\frac{23}{32}} \lesssim ||g||_{L^{\frac{32}{23}-\epsilon}(F^3)} $$
for all $\epsilon >0$ and all characteristic-type regular functions $g$. By Lemma \ref{lem:dyPig} and Lemma \ref{lem:epsRem}, this implies the dual estimate for every exponent $p<\frac{32}{23}$. Dualizing, this is Theorem \ref{thm:Main}.

\section{Additional Remarks}
We conclude with a few additional remarks:
\begin{itemize}

\item We expect that there is a higher dimensional analog of this argument, but we do not explore that here.

\item The correction above changes the role of the $m^{7/2}$ term in Proposition \ref{prop:TboundGF}. In the previous version, this term was not treated as the dominant term in the final range, and we remarked that improving it would not affect the main theorem. The corrected bookkeeping shows that this term is in fact the obstruction in the middle range. Thus the observation that this term can be improved in certain parameter ranges, for example by bringing in stronger incidence estimates in prime order fields, is directly relevant. These refinements, and the corresponding improved restriction estimates, will appear in forthcoming work.

\item It seems that a leading term of size near $m^3$ (for certain parameter ranges) would be the best one could hope for in Proposition \ref{prop:TboundGF}, and that such an estimate would not yield an exponent better than $r \geq 10/3$ in the main theorem via this argument. This is the same range that would follow from an optimal Szemer\'edi-Trotter-type estimate using the prior machinery, as observed by Rudnev and Shkredov \cite{RS}.

\end{itemize}

\section{Appendix: The $L^4$ estimate}
The purpose of this appendix is to give a short proof of \eqref{eq:gzToR} for the sake of self-containedness. This is not new and is used implicitly in much of the prior work, however it is not stated precisely in this form in a way we can directly refer the reader to.

\begin{lemma}\label{lem:Rrel} Let $G:F^3 \rightarrow \mathbb{C}$ be a function supported on a vertical hyperplane. In other words let $G(\underline{x},t):=\delta(t-z)g(\underline{x})$ for some function $g:F^2 \rightarrow \mathbb{C}$. Moreover assume that $g \sim 1$ on its support $A$. In the notation above one has 
$$|| G * K||_{L^4(F^3)}^{4} \lesssim |F|^{-1}\mathcal{R}(A).$$
\end{lemma}
\begin{proof}
By translation we may assume $z=0$. Then
$$ \sum_{x \in F^3}|{G} * K(x)|^4 =  |F|^{-4}  \sum_{\underline{x} \in F^{2}}\sum_{\substack {x_3 \in F \\ x_3 \neq 0}} \left|\sum_{y \in F^{2}} g(y) e\left( (\underline{x}-y)\cdot(\underline{x}-y)/(-4x_3) \right)\right|^4.$$  
Noting that
$$(\underline{x}-y)\cdot(\underline{x}-y)/(-4x_3) = (\underline{x}\cdot \underline{x} -2\underline{x}\cdot y + y\cdot y)/(-4x_3)$$
we performing the change of variables $t = 1/(-4 x_3)$ and $u = \underline{x}/(2x_3)$, so that
$$(\underline{x}-y)\cdot(\underline{x}-y)/(-4x_3) = u\cdot u/(4t) +u\cdot y +t y\cdot y.$$
We note that this change of variables is sometimes called the pseudo-conformal transformation in related literature. Using this, the left-hand side above is equal to
$$ |F|^{-4}  \sum_{u \in F^{2}}\sum_{\substack {t \in F \\ t \neq 0}} \left|e\left(\frac{u \cdot u}{4t} \right) \sum_{y \in F^{2}} g(y) e\left( u \cdot y + t y\cdot y  \right)\right|^4$$  
$$ =|F|^{-4}  \sum_{u \in F^{2}}\sum_{\substack {t \in F \\ t \neq 0}} \left| \sum_{y \in F^{2}} g(y) e\left( u \cdot y + t y\cdot y  \right)\right|^4.$$
Expanding the $L^4$ norm and adding the omitted $t=0$ term only increases the expression up to an acceptable absolute constant. We obtain the upper bound
$$\lesssim |F|^{-1} |\{ x,y,z,w \in A : x+y=z+w \text{ and } x\cdot x + y\cdot y = z\cdot z + w \cdot w \}|. $$
The reader might recognize this as the ``additive energy" of a subset of the paraboloid. Solving for $x$ in the first relation (that is $x=z+w-y$), combined with the second relation, states
$$(z+w-y)\cdot (z+w-y) = z\cdot z + w \cdot w- y\cdot y$$
which holds if and only if 
$$(z - y)\cdot(y-w) =0$$
which, in the notation above, is equivalent to $(z,y,w)$ forming a corner. Repeating the argument, but solving for $y,z$ and $w$, respectively, establishes that these quadruples are bounded by the number of rectangles in $A$.
\end{proof}

 \bibliographystyle{amsplain}

\end{document}